\documentclass{amsart}
\usepackage{amssymb, amsmath, amsthm, mathrsfs}
\usepackage{cite, hyperref}
\makeatletter
\renewcommand\@makefntext[1]{\noindent #1}
\makeatother
\usepackage[headsep=0.25truein, top=1.25truein, bottom=1truein, hmargin={1.3truein,1.3truein}]{geometry}
\allowdisplaybreaks[4]
\numberwithin{equation}{section}

\usepackage{graphicx,xcolor}
\usepackage[all]{xy}
\definecolor{db}{RGB}{23,20,219}
\definecolor{dg}{RGB}{2,101,15}
\colorlet{sectitlecolor}{red!60!black}
\colorlet{sectboxcolor}{cyan!30}
\colorlet{secnumcolor}{orange}
\usepackage[explicit]{titlesec}
\titleformat{\section}
{\sffamily\color{sectitlecolor}\Large\bfseries\filcenter}{}{2em}{\thesection.\quad #1}%
\newtheoremstyle{mytheorem}{5pt}{}{\color{db}}{}{\color{db}\bfseries}{}{ }{}
\titleformat{\subsection}
{\sffamily\color{sectitlecolor}\bfseries\filcenter}{}{2em}{\thesubsection.\quad #1}%
\newtheoremstyle{mytheorem}{5pt}{}{\color{db}}{}{\color{db}\bfseries}{}{ }{}
\theoremstyle{mytheorem}

\newtheorem{theorem}{Theorem}[section]
\newtheorem{corollary}[theorem]{Corollary}

\newtheorem{proposition}[theorem]{Proposition}
\theoremstyle{definition}
\newtheorem{definition}[theorem]{Definition}
\theoremstyle{example}
\newtheorem{example}{Example}
\theoremstyle{remark}

\numberwithin{equation}{section}

\usepackage[displaymath, mathlines]{lineno} 

\newcommand{\bm}[1]{\boldsymbol{#1}}

\def\dual{{\sharp}}

\newcommand{\ZU}[2]{%
Z_U\left(\begin{matrix}{#1}\\%
{#2}\end{matrix}\right)}
\newcommand{\ZL}[2]{%
Z_L\left(\begin{matrix}{#1}\\%
{#2}\end{matrix}\right)}
\newcommand{\ZB}[2]{%
Z_B\left(\begin{matrix}{#1}\\%
{#2}\end{matrix}\right)}

\newcommand{\frP}{\mathfrak{P}}
\newcommand{\Q}{\mathbb{Q}}
\newcommand{\R}{\mathbb{R}}

 \DeclareSymbolFont{Shuffle}{U}{shuf}{m}{n}
\DeclareFontFamily{U}{shuf}{}
\DeclareFontShape{U}{shuf}{m}{n}{%
  <-8>shuffle7%
  <8->shuffle10%
}{}
\DeclareMathSymbol\sh{\mathbin}{Shuffle}{"001}
\DeclareMathSymbol\bsh{\mathbin}{Shuffle}{"002}

\makeatletter
\def\rightharpoonupfill@{\arrowfill@\relbar\relbar\rightharpoonup}
\newcommand{\oright}[3]{%
\mathpalette{\overarrow@\rightharpoonupfill@}{\bm{#1}_{#2,#3}}}
\def\leftharpoonupfill@{\arrowfill@\leftharpoonup\relbar\relbar}
\newcommand{\oleft}[3]{%
\mathpalette{\overarrow@\leftharpoonupfill@}{\bm{#1}_{#2,#3}}}
\makeatother

\begin{document}
\footnotetext{%
Date: 2023-04-14; Version 11. 
}

\title{Some Symmetry and Duality Theorems on Multiple 
Zeta(-star) Values}

\author{Kwang-Wu Chen}
\address{Department of Mathematics, University of Taipei, 100234 Taipei, Taiwan}
\email{kwchen@utaipei.edu.tw}
\thanks{The first author (corresponding author)
was funded by the Ministry of Science and Technology,
Taiwan, R.O.C., under Grant MOST 111-2115-M-845-001.}

\author{Minking Eie}
\address{Department of Mathematics, National Chung Cheng University, 168 University Road, Min-Hsiung, Chia-Yi 62145, Taiwan}
\email{minkingeie@gmail.com}

\author{Yao Lin Ong}
\address{Executive Master of Business Administration\\
Chang Jung Christian University\\
No.1, Changda Rd., Gueiren District, Tainan City 71101, Taiwan (Yao Lin Ong)
}
\email{ylong@mail.cjcu.edu.tw}


\begin{abstract}\normalsize
In this paper, we provide a symmetric formula and 
a duality formula relating multiple zeta values and zeta-star values.
Leveraging Zagier's formula for computing $\zeta^\star(\{2\}^p,3,\{2\}^q)$, 
we employ our theorems to establish a formula for computing 
$\zeta^\star(\{2\}^p,1,\{2\}^q)$ for 
any positive integers $p$ and $q$, along with other formulas of interest.
\end{abstract}



\maketitle

\section{Introduction}\label{sec1}
For an $r$-tuple $\boldsymbol{\alpha} = (\alpha_{1}, \alpha_{2}, \ldots, \alpha_{r})$ 
of positive integers with $\alpha_{r} \geq 2$, a multiple zeta value $\zeta(\bm\alpha)$
and a multiple zeta-star value $\zeta^\star(\bm\alpha)$
are defined to be \cite{E09, E13, Hoffman1992, Ohno2005, Zagier1994}

\begin{align*}
  \zeta(\boldsymbol{\alpha})
  &= \sum_{1 \leq k_{1} < k_{2} < \cdots < k_{r}} k_{1}^{-\alpha_{1}}
    k_{2}^{-\alpha_{2}} \cdots k_{r}^{-\alpha_{r}},
\quad\mbox{and}\\
  \zeta^{\star}(\boldsymbol{\alpha})
  &= \sum_{1 \leq k_{1} \leq k_{2} \leq \cdots \leq k_{r}} k_{1}^{-\alpha_{1}}
    k_{2}^{-\alpha_{2}} \cdots k_{r}^{-\alpha_{r}}.
\end{align*}

\noindent We denote the parameters $w(\bm\alpha)=|\bm\alpha| 
= \alpha_{1} + \alpha_{2} + \cdots + \alpha_{r}$,
$d(\bm\alpha)=r$, and $h(\bm\alpha)=\#\{i\mid \alpha_i>1, 1\leq i\leq r\}$,
called respectively the weight, the depth, and the height of $\bm\alpha$
(or of $\zeta(\bm\alpha)$, or of $\zeta^\star(\bm\alpha)$).

Yamamoto \cite{Yama2017} introduced a combinatorial generalization of 
the iterated integral, the integral associated with a $2$-poset.
Kaneko and Yamamoto \cite{KY2018} conjecture that the following 
integral-series identity is sufficient to describe all 
linear relations of multiple zeta values over $\mathbb Q$. 

\[
I\left(\begin{xy}
{(-1,-3) \ar@{{*}.o} |-{\bm\alpha} (7,5)},
{(-1,-2) \ar@/^2mm/ @{-} (6,5)},
{(7,5) \ar@{-}@[red] (9,2)},
{(9,2) \ar@{-} (13,6)},
{(9,2) \ar@{{*}o} |-{\bm\beta} (21,6)},
{(9,2) \ar@{.} (17,2)},
{(13,6) \ar@{.} (21,6)},
{(17,2) \ar@{-} (21,6)},
\end{xy}\ \right)
=\ZU{\bm\alpha}{\bm\beta}, 
\]

\noindent where the left-hand side of the above identity 
is an integral associated with a $2$-posets 
introduced by Yamamoto \cite{Yama2017}
(we will introduce the associated notations in Section 2).

The use of Yamamoto's integral is prevalent among 
researchers studying multiple zeta values and multiple zeta-star values
\cite{HMO2021, KY2018, Yama2020}. 

Many scholars have conducted research on the function $Z_U$ 
appearing on the right-hand side of the above identity
\cite{Chen2017,ChenE2022,KY2018, NPY2018,NT2022}.

In this paper, we investigate 
the functions $Z_U$ (Schur multiple zeta values of anti-hook type),
$Z_L$ (Schur multiple zeta values of hook type), and $Z_B$ 
\cite{Chen2017, CCE2016, ChenE2022}:
\begin{align*}
\ZU{\bm\alpha}{\bm\beta}
&:=\!\!\!\!\!\!
\sum_{1\leq k_1<k_2<\cdots<k_r}k_1^{-\alpha_1}k_2^{-\alpha_2}\cdots k_r^{-\alpha_r}
\sum_{1\leq\ell_1\leq\ell_2\leq\cdots\leq\ell_m\leq k_r}
\ell_1^{-\beta_1}\ell_2^{-\beta_2}\cdots\ell_m^{-\beta_m},\\
\ZL{\bm\alpha}{\bm\beta}
&:=\!\!\!\!\!\!
\sum_{1\leq k_1<k_2<\cdots<k_r}k_1^{-\alpha_1}k_2^{-\alpha_2}\cdots k_r^{-\alpha_r}
\sum_{k_1\leq\ell_1\leq\ell_2\leq\cdots\leq\ell_m}
\ell_1^{-\beta_1}\ell_2^{-\beta_2}\cdots\ell_m^{-\beta_m},\\
\ZB{\bm\alpha}{\bm\beta}
&:=\!\!\!\!\!\!
\sum_{1\leq k_1<k_2<\cdots<k_r}k_1^{-\alpha_1}k_2^{-\alpha_2}\cdots k_r^{-\alpha_r}
\sum_{k_1\leq\ell_1\leq\ell_2\leq\cdots\leq\ell_m\leq k_r}
\ell_1^{-\beta_1}\ell_2^{-\beta_2}\cdots\ell_m^{-\beta_m}.
\end{align*}
For the sake of convergence, the functions $Z_U$ and $Z_B$ are restricted by 
$\alpha_r\geq 2$, and the function $Z_L$ is restricted by $\alpha_r\geq 2, \beta_m\geq 2$.

We use Yamamoto's integral to obtain the following 
symmetric form between multiple zeta(-star) values.
We simplify our notations with 
$$
\oright{\alpha}{i}{j}=(\alpha_{i+1},\alpha_2,\ldots,\alpha_j), \quad
\oleft{\alpha}{i}{j}=(\alpha_j,\alpha_{j-1},\ldots,\alpha_{i+1}),
$$
for $i<j$. Note that if $i\geq j$, we set 
$\oright{\alpha}{i}{j}=\oleft{\alpha}{i}{j}=\emptyset$.
For brevity, we use the lowercase English letters and lowercase
Greek letters, with or without subscripts, 
in summations to represent non-negative 
and positive integers, unless otherwise specified.

\begin{theorem}\label{thm.x111}
Let $r,m,\theta$ be nonnegative integers with $\theta\geq 2$, 
and the vectors $\bm\alpha=(\alpha_1,\alpha_2,\ldots,\alpha_r)$,
$\bm\beta=(\beta_1,\beta_2,\ldots,\beta_m)$ with $\alpha_i\geq 2$, $\beta_m\geq 2$,
for $1\leq i\leq r$. Then
\begin{align}    \label{eq.x4.1}
\sum_{a+b=r\atop c+d=m}
(-1)^{a+c}\zeta^\star(\oright{\alpha}{0}{b})
\zeta(\oleft{\beta}{0}{c},\theta,\oleft{\alpha}{b}{r})
\zeta^\star(\oright{\beta}{c}{m})
&=\zeta^\star(\bm\alpha,\theta,\bm\beta),\\
\sum_{a+b=r\atop c+d=m}    \label{eq.x4.2}
(-1)^{a+c}\zeta(\oright{\alpha}{0}{b})
\zeta^\star(\oleft{\beta}{0}{c},\theta,\oleft{\alpha}{b}{r})
\zeta(\oright{\beta}{c}{m})
&=\zeta(\bm\alpha,\theta,\bm\beta).
\end{align}
\end{theorem}

The following symmetric formula that we have is truly exquisite.
\begin{corollary}
For any nonnegative integers $p,q,s,t,\alpha$ with $s\geq 2,t\geq 2,\alpha\geq 2$, we have
\begin{align*}
\sum_{a+b=q\atop c+d=p}
(-1)^{a+c}\zeta^\star(\{s\}^b)\zeta(\{s\}^a,\alpha,\{t\}^c)
\zeta^\star(\{t\}^d)&=\zeta^\star(\{t\}^p,\alpha,\{s\}^q),\\
\sum_{a+b=q\atop c+d=p}
(-1)^{a+c}\zeta(\{s\}^b)\zeta^\star(\{s\}^a,\alpha,\{t\}^c)
\zeta(\{t\}^d)&=\zeta(\{t\}^p,\alpha,\{s\}^q).
\end{align*}
\end{corollary}
We note that a version of finite sums of the above two identities can be found in \cite{TY2013, Xu2019}.

Let $(a_1,b_1)$, $(a_2,b_2)$, $\ldots$, $(a_n,b_n)$ be $n$ pairs of 
nonnegative integers. If we write 
$$
(\alpha_1,\alpha_2,\ldots,\alpha_r)
=(\{1\}^{a_1},b_1+2,\{1\}^{a_2},b_2+2,\ldots,\{1\}^{a_n},b_n+2)
$$
and set 
$$
(\{1\}^{b_n},a_n+2,\{1\}^{b_{n-1}},a_{n-1}+2,\ldots,\{1\}^{b_1},a_1+2)
=(\beta_1,\beta_2,\ldots,\beta_m),
$$
then the duality theorem of multiple zeta values \cite{Ohno1999} is stated as 
$$
\zeta(\alpha_1,\alpha_2,\ldots,\alpha_r)
=\zeta(\beta_1,\beta_2,\ldots,\beta_m).
$$
We say that $\zeta(\beta_1,\ldots,\beta_m)$ is the dual of 
$\zeta(\alpha_1,\ldots,\alpha_r)$, and we denote it as $
\zeta(\alpha_1,\ldots,\alpha_r)^\dual=\zeta(\beta_1,\ldots,\beta_m).
$

In 2015, the first author \cite{Chen2017} gave a formula related to $Z_U$:
$$
\ZU{\beta_1,\beta_2,\ldots,\beta_m}{\{1\}^q}
=\sum_{|\bm d|=q}\zeta(\alpha_1+d_1,\ldots,\alpha_r+d_r)\prod^r_{j=1}\binom{\alpha_j+d_j-1}{d_j},
$$
if $\zeta(\beta_1,\ldots,\beta_m)$ is the dual of $\zeta(\alpha_1,\ldots,\alpha_r)$.

In this paper, we present three different methods to prove 
the following duality theorem between these $Z_B$, $Z_U$, and $Z_L$ functions.
\begin{theorem}
If $\zeta(\beta_1,\beta_2,\ldots,\beta_m)$ is the dual of $\zeta(\alpha_1,\alpha_2,\ldots,\alpha_r)$.
Given any nonnegative integer $n$, we have
\begin{align*}
\ZU{\alpha_1,\alpha_2,\ldots,\alpha_r}{\{2\}^n}
&=\ZL{\beta_1,\beta_2,\ldots,\beta_m}{\{2\}^n},
 \quad\mbox{and}\\
\ZB{\alpha_1,\alpha_2,\ldots,\alpha_r}{\{2\}^n}
&=\ZB{\beta_1,\beta_2,\ldots,\beta_m}{\{2\}^n}.
\end{align*}
\end{theorem}

The authors \cite{ChenE2022} use these three functions to give three new sum formulas
for multiple zeta(-star) values with height $\leq 2$ and the evaluation of 
$\zeta^\star(\{1\}^m,\{2\}^{n+1})$. For example, 
the authors use $Z_B$ to obtain \cite[Theorem 1.1]{ChenE2022}
$$
\sum_{\alpha_1+\alpha_2=n\atop \alpha_1,\alpha_2\geq 1}
\zeta^\star(\alpha_1,\{1\}^m,\alpha_2+1)=(m+n)\zeta(m+n+1).
$$

To make the symbols more compact, we adopt the fact $\zeta(\emptyset)=\zeta^\star(\emptyset)=1$.
Moreover, we use the duality theorem on the functions
$Z_U$ and $Z_L$ to get the following formula.

\begin{theorem}
For any nonnegative integers $p$, $q$, and $r$ with $p>0$, $q>0$, we have
\begin{align*}
&\zeta^\star(\{2\}^p,\{1\}^r,\{2\}^q)+(-1)^{r+1}\zeta^\star(\{2\}^q,r+2,\{2\}^{p-1}) \\
&=\sum_{a+b=r-1}(-1)^a\zeta^\star(a+2,\{2\}^{p-1})
\zeta^\star(\{1\}^{b+1},\{2\}^q).
\end{align*}
\end{theorem}

Leveraging Zagier's formula (ref. \cite{Zagier2012})
for computing $\zeta^\star(\{2\}^p,3,\{2\}^q)$, 
we employ our theorems to establish a formula for computing 
$\zeta^\star(\{2\}^p,1,\{2\}^q)$:
$$
\zeta^\star(\{2\}^p,1,\{2\}^q) =
2\sum^{p+q}_{k=1}
\left[\binom{2k}{2q}-\left(1-\frac{1}{2^{2k}}\right)\binom{2k}{2p-1}\right]
\zeta^\star(\{2\}^{p+q-k})\zeta(2k+1),
$$
for any positive integers $p$ and $q$, along with other formulas of interest 
in the final section.

\section{Some Preliminaries and Auxiliary Tools}\label{sec.2}
In this section, we review the definitions and basic properties of 2-labeled posets 
(in this paper, we call them 2-posets for short) 
and the associated integrals first introduced by Yamamoto \cite{Yama2017}. 
\begin{definition}\cite[Definition 3.1]{KY2018}
A \textit{2-poset} is a pair $(X,\delta_X)$, where $X=(X,\leq)$ is 
a finite partially ordered set (poset for short) and 
$\delta_X$ is a map from $X$ to $\{0,1\}$. 
We often omit  $\delta_X$ and simply say ``a 2-poset $X$.'' 
The $\delta_X$ is called the \emph{label map} of $X$. 

A 2-poset $(X,\delta_X)$ is called \textit{admissible} if 
$\delta_X(x)=0$ for all maximal elements $x\in X$ and 
$\delta_X(x)=1$ for all minimal elements $x\in X$. 
\end{definition}

A 2-poset is depicted as a Hasse diagram in which an element $x$ with 
$\delta(x)=0$ (resp. $\delta(x)=1$) is represented by $\circ$ 
(resp. $\bullet$). For example, the diagram 
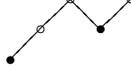
\[\begin{xy}
{(0,-4) \ar @{{*}-o} (4,0)},
{(4,0) \ar @{-o} (8,4)},
{(8,4) \ar @[red]@{-} (12,0)},
{(12,0) \ar @{{*}-o} (16,4)}
\end{xy} \]
represents the 2-poset $X=\{x_1,x_2,x_3,x_4,x_5\}$ with order 
$x_1<x_2<x_3>x_4<x_5$ and label 
$(\delta_X(x_1),\ldots,\delta_X(x_5))=(1,0,0,1,0)$. 

\begin{definition}\cite[Definition 3.2]{KY2018}
For an admissible 2-poset $X$, we define the associated integral 
\begin{equation}\label{eq:I(X)}
I(X)=\int_{\Delta_X}\prod_{x\in X}\omega_{\delta_X(x)}(t_x), 
\end{equation}
where 
\[\Delta_X=\bigl\{(t_x)_x\in [0,1]^X \bigm| t_x<t_y \text{ if } x<y\bigr\}
\quad\mbox{and }\quad
\omega_0(t)=\frac{dt}{t}, \quad \omega_1(t)=\frac{dt}{1-t}. \]
\end{definition}
Note that the admissibility of a 2-poset corresponds to 
the convergence of the associated integral. 

\begin{example}\label{ex:MZV}
When an admissible 2-poset is totally ordered, 
the corresponding integral is exactly the iterated integral expression 
for a multiple zeta value. 
To be precise, for an index 
$\bm\alpha=(\alpha_1,\ldots,\alpha_r)$ (admissible or not), 
we write the `totally ordered' diagram: 
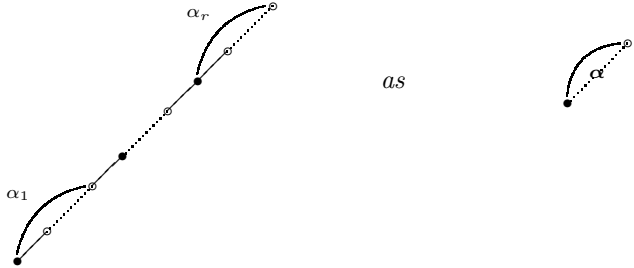
\[\begin{xy}
{(0,-24) \ar @{{*}-o} (4,-20)}, 
{(4,-20) \ar @{.o} (10,-14)}, 
{(10,-14) \ar @[red] @{{.}-} (14,-10)}, 
{(14,-10) \ar @{{*}.o} (20,-4)}, 
{(20,-4) \ar @[red] @{-} (24,0)}, 
{(24,0) \ar @{{*}-o} (28,4)}, 
{(28,4) \ar @{.o} (34,10)}, 
{(0,-23) \ar @/^2mm/ @{-}^{\alpha_1} (9,-14)}, 
{(24,1) \ar @/^2mm/ @{-}^{\alpha_r} (33,10)} 
\end{xy} \qquad\qquad\mbox{as}\qquad\qquad\qquad
\begin{xy}
{(0,-3) \ar@{{*}.o} |-{\bm\alpha} (8,5)},
{(0,-2) \ar@/^2mm/ @{-} (7,5)},
\end{xy}\]

In \cite{Yama2017}, an integral expression for multiple zeta-star values 
is described in terms of a 2-poset. 
For an index $\bm\beta=(\beta_1,\beta_2,\ldots,\beta_m)$, we write the following diagram:
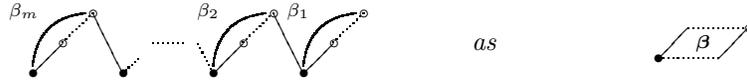
\[\begin{xy}
{(0,-4) \ar @{{*}-o} (4,0)}, 
{(4,0) \ar @{.o} (8,4)}, 
{(8,4) \ar @[red] @{-} (12,-4)}, 
{(12,-4) \ar @{{*}.} (14,-2)}, 
{(16,0) \ar @{.} (20,0)}, 
{(22,0) \ar @{.{*}} (24,-4)}, 
{(24,-4) \ar @{-{o}} (28,0)}, 
{(28,0) \ar @{.o} (32,4)}, 
{(32,4) \ar @[red] @{-} (36,-4)}, 
{(36,-4) \ar @{{*}-{o}} (40,0)}, 
{(40,0) \ar @{.o} (44,4)}, 
{(0,-3) \ar @/^2mm/ @{-}^{\beta_m} (7,4)}, 
{(24,-3) \ar @/^2mm/ @{-}^{\beta_2} (31,4)}, 
{(36,-3) \ar @/^2mm/ @{-}^{\beta_1} (43,4)}, 
\end{xy}
\qquad\qquad\mbox{as}\qquad\qquad\qquad
\begin{xy}
{(0,-2) \ar@{-} (4,2)},
{(0,-2) \ar@{{*}o} |-{\bm\beta} (12,2)},
{(0,-2) \ar@{.} (8,-2)},
{(4,2) \ar@{.} (12,2)},
{(8,-2) \ar@{-} (12,2)},
\end{xy}
\]

Then, if $\bm\alpha$ and $\bm\beta$ are admissible, we have 
\cite[Proposition 2.4, 2.7]{Yama2020}
\begin{equation}\label{eq:SI MZV}
\zeta(\bm\alpha)=
I\left(\begin{xy}
{(0,-3) \ar@{{*}.o} |-{\bm\alpha} (8,5)},
{(0,-2) \ar@/^2mm/ @{-} (7,5)},
\end{xy}\right),
\qquad\qquad\mbox{and}\qquad\qquad
\zeta^\star(\bm\beta)=
I\left(\ \begin{xy}
{(0,-2) \ar@{-} (4,2)},
{(0,-2) \ar@{{*}o} |-{\bm\beta} (12,2)},
{(0,-2) \ar@{.} (8,-2)},
{(4,2) \ar@{.} (12,2)},
{(8,-2) \ar@{-} (12,2)},
\end{xy}\ \right). 
\end{equation}
\end{example}

For example, 
\[
\zeta(1,2)=
I\left(\begin{xy}
{(0,-3) \ar@{{*}.o} |-{\bm(1,2)} (8,5)},
{(0,-2) \ar@/^3mm/ @{-} (7,5)},
\end{xy}\,\right)
\ =\
I\left(\begin{xy}
{(0,-3) \ar@{{*}-{*}}^1 (3,0)},
{(3,0) \ar@{-o} (6,3)},
{(3,0) \ar@/_2mm/ @{-}_2 (6,3)},
\end{xy}\right),
\qquad
\zeta^\star(2,3)=
I\left(\begin{xy}
{(0,-2) \ar@{-} (4,4)},
{(0,-2) \ar@{{*}o} |-{\bm(2,3)} (13,4)},
{(0,-2) \ar@{.} (9,-2)},
{(4,4) \ar@{.} (13,4)},
{(9,-2) \ar@{-} (13,4)}, 
\end{xy}\,\,\right)
\ =\ 
I\left(\begin{xy}
{(0,-4) \ar @{{*}-} (4,0)}, 
{(4,0) \ar @{o-} (8,4)}, 
{(8,4) \ar @{o-} (12,0)}, 
{(12,0) \ar @{{*}-o} (16,4)}, 
{(0,-4) \ar@/^2mm/@{-}^3 (8,4)},
{(12,0) \ar@/^2mm/@{-}^2 (16,4)},
\end{xy}\right)\ .
\]

We also recall an algebraic setup for 2-posets 
(cf.\ Remark at the end of \S2 of \cite{Yama2017}). 
Let $\frP$ be the $\Q$-algebra generated by 
the isomorphism classes of 2-posets, 
whose multiplication is given by the disjoint union of 2-posets. 
Then the integral \eqref{eq:I(X)} defines a $\Q$-algebra homomorphism 
$I\colon\frP^0\to\R$ from the subalgebra $\frP^0$ of $\frP$ generated by 
the classes of admissible 2-posets. 

Following these notations, we have 
\begin{equation}
\ZU{\bm\alpha}{\bm\beta}=
I\left(\begin{xy}
{(-1,-3) \ar@{{*}.o} |-{\bm\alpha} (7,5)},
{(-1,-2) \ar@/^2mm/ @{-} (6,5)},
{(7,5) \ar@{-}@[red] (9,2)},
{(9,2) \ar@{-} (13,6)},
{(9,2) \ar@{{*}o} |-{\bm\beta} (21,6)},
{(9,2) \ar@{.} (17,2)},
{(13,6) \ar@{.} (21,6)},
{(17,2) \ar@{-} (21,6)},
\end{xy}\ \right),
\qquad\mbox{and}\qquad
\ZL{\bm\alpha}{\bm\beta}=I\left(\ \begin{xy}
{(0,-5) \ar@{-} (4,-1)},
{(0,-5) \ar@{{*}o} |-{\bm\beta} (12,-1)},
{(0,-5) \ar@{.} (8,-5)},
{(4,-1) \ar@{.} (12,-1)},
{(8,-5) \ar@{-} (12,-1)},
{(12,-1) \ar@[red]@{-} (15,-4)},
{(15,-4) \ar@{{*}.o} |-{\bm\alpha} (23,4)},
{(16,-4) \ar@/_2mm/ @{-} (23,3)},
\end{xy}\ \ \right).
\end{equation}

In fact, the upper left equation is the same as \cite[Theorem 4.1]{KY2018}.
Kaneko and Yamamoto conjecture that any linear dependency of MZVs over $\mathbb Q$
is deduced from this equation. 

We evaluate $\ZL{1,2}{1,2}$ as an example:
Let $D=\{(t_1,t_2,t_3,t_4,t_5,t_6)\in[0,1]^6\mid t_1<t_2>t_3>t_4<t_5<t_6\}$, we calculate
\begin{align*}
I\left(  
\begin{xy}
{(0,0) \ar @{{*}-} (4,4)}, 
{(4,4) \ar @{o-} (8,0)}, 
{(8,0) \ar @{{*}-} (12,-4)}, 
{(12,-4) \ar @{{*}-} (16,0)}, 
{(16,0) \ar @{{*}-o} (20,4)}, 
\end{xy}\  
\right) &= \int_D\frac{dt_1}{1-t_1}\frac{dt_2}{t_2}\frac{dt_3}{1-t_3}\frac{dt_4}{1-t_4}
\frac{dt_5}{1-t_5}\frac{dt_6}{t_6} \\
&=\int_{D-\{t_1\}}\left(\int^{t_2}_0\sum^\infty_{n=0}t^n_1\,dt_1\right) 
\frac{dt_2}{t_2}\frac{dt_3}{1-t_3}\frac{dt_4}{1-t_4}
\frac{dt_5}{1-t_5}\frac{dt_6}{t_6} \\
&=\int_{D-\{t_1,t_2\}}\left(\int^1_{t_3}\sum^\infty_{\ell_2=1}\frac{t_2^{\ell_2-1}}{\ell_2} \,dt_2\right)
\frac{dt_3}{1-t_3}\frac{dt_4}{1-t_4}
\frac{dt_5}{1-t_5}\frac{dt_6}{t_6} \\
&=\sum^\infty_{\ell_2=1}\frac{1}{\ell_2^2} \int_{D-\{t_1,t_2,t_3\}}
\left(\int^1_{t_4}\sum^{\ell_2-1}_{n=0}t_3^n\,dt_3\right)
\frac{dt_4}{1-t_4}\frac{dt_5}{1-t_5}\frac{dt_6}{t_6} \\
&=\sum^\infty_{\ell_2=1}\frac{1}{\ell_2^2} \sum^{\ell_2}_{\ell_1=1}\frac{1}{\ell_1}
\int_{0<t_5<t_6<1}
\left(\int^{t_5}_0\sum^{\ell_1-1}_{n=0}t_4^n\,dt_4\right)
\frac{dt_5}{1-t_5}\frac{dt_6}{t_6} \\
&=\sum^\infty_{\ell_2=1}\frac{1}{\ell_2^2} \sum^{\ell_2}_{\ell_1=1}\frac{1}{\ell_1}
\sum^{\ell_1}_{k_1=1}\frac{1}{k_1} \int_0^1
\left(\int^{t_6}_0\sum^\infty_{n=0}t_5^{n+k_1}\,dt_5\right)\frac{dt_6}{t_6} \\
&=\sum^\infty_{\ell_2=1}\frac{1}{\ell_2^2} \sum^{\ell_2}_{\ell_1=1}\frac{1}{\ell_1}
\sum^{\ell_1}_{k_1=1}\frac{1}{k_1} 
\sum^\infty_{n=1}\frac{1}{(n+k_1)^2}  \ =\ZL{1,2}{1,2}. 
\end{align*}

We need some properties of these three functions $Z_U$, $Z_L$ and $Z_B$, 
please see \cite{ChenE2022} for details:
\begin{align}\label{eq.zlrec}
\zeta(\bm\alpha)\zeta^\star(\bm\beta) 
&=\ZL{\alpha_1,\alpha_2,\ldots,\alpha_r}{\beta_1,\beta_2,\ldots,\beta_m}
+\ZL{\beta_1,\alpha_1,\ldots,\alpha_r}{\beta_2,\beta_3,\ldots,\beta_m},\\
\zeta(\bm\alpha)\zeta^\star(\bm\beta)     \label{eq.zurec}
&=\ZU{\alpha_1,\alpha_2,\ldots,\alpha_r}{\beta_1,\beta_2,\ldots,\beta_m}
+\ZU{\alpha_1,\ldots,\alpha_r,\beta_m}{\beta_1,\ldots,\beta_{m-1}}.
\end{align}
\begin{proposition}\cite[Proposition 4.1]{ChenE2022}
For an $r$-tuple $\bm\alpha=(\alpha_1,\alpha_2,\ldots,\alpha_r)$ 
of positive integers with $\alpha_1\geq 2$, $\alpha_r\geq 2$, we have
\begin{equation}\label{eq.2.3}
\sum^{r}_{k=0}(-1)^{k}\zeta^\star(\alpha_k,\alpha_{k-1},\ldots,\alpha_1)
\zeta(\alpha_{k+1},\alpha_{k+2},\ldots,\alpha_r)
=\left\{\begin{array}{ll}1,&\mbox{ if }r=0,\\
0,&\mbox{ otherwise.}\end{array}\right.
\end{equation}
\end{proposition}

Let $\bm\alpha=(s,\ldots,s)=(\{s\}^r)$ in Eq.\,(\ref{eq.2.3}). 
We obtain the classical known result \cite[Eq.\,(5.8)]{ChenE2022}
\begin{equation}\label{eq.2.4}
\sum_{a+b=r}(-1)^a\zeta(\{s\}^a)\zeta^\star(\{s\}^b)
=\left\{\begin{array}{ll}1,&\mbox{ if }r=0,\\
0, & \mbox{ otherwise.}\end{array}\right.
\end{equation}

We derive an evaluation of $Z_L$  using Eq.\,(\ref{eq.zlrec}).
\begin{align*}
\ZL{\alpha_1,\ldots,\alpha_r}{\beta_1,\ldots,\beta_m}
&=\zeta(\alpha_1,\ldots,\alpha_r)\zeta^\star(\beta_1,\ldots,\beta_m)
-\ZL{\beta_1,\alpha_1,\ldots,\alpha_r}{\beta_2,\ldots,\beta_m} \\
&=\zeta(\alpha_1,\ldots,\alpha_r)\zeta^\star(\beta_1,\ldots,\beta_m)
-\zeta(\beta_1,\alpha_1,\ldots, \alpha_r)\zeta^\star(\beta_2,\ldots,\beta_m) \\
&\qquad\qquad\qquad\qquad
+(-1)^2\ZL{\beta_2,\beta_1,\alpha_1,\ldots,\alpha_r}{\beta_3,\ldots,\beta_m}.
\end{align*}
We use Eq.\,(\ref{eq.zlrec}) repeatedly until the vector in the second row is the empty set.
Since 
$$
\ZL{\beta_m,\beta_{m-1},\ldots,\beta_1,\alpha_1,\ldots,\alpha_r}{\emptyset}
=\zeta(\beta_m,\beta_{m-1},\ldots,\beta_1,\alpha_1,\ldots,\alpha_r),
$$
we arrive a formula for the function $Z_L$:
$$
\ZL{\bm\alpha}{\bm\beta}
=\sum_{a+b=m}(-1)^a\zeta(\beta_a,\ldots,\beta_1,\alpha_1,\ldots,\alpha_r)
\zeta^\star(\beta_{a+1},\ldots,\beta_m).
$$

If we use Eq.\,(\ref{eq.zlrec}) in a different direction
$$
\ZL{\alpha_1,\ldots,\alpha_r}{\beta_1,\ldots,\beta_m}
=\zeta(\alpha_2,\ldots,\alpha_r)\zeta^\star(\alpha_1,\beta_1,\ldots,\beta_m)
-\ZL{\alpha_2,\ldots,\alpha_r}{\alpha_1,\beta_1,\beta_2,\ldots,\beta_m},
$$
 then we obtain the following 
formula for $Z_L$.
$$
\ZL{\alpha_1,\ldots,\alpha_r}{\beta_1,\ldots,\beta_m}
=\sum_{a+b=r-1}(-1)^a\zeta(\alpha_{a+2},\ldots,\alpha_r)
\zeta^\star(\alpha_{a+1},\ldots,\alpha_1,\beta_1,\ldots,\beta_m).  
$$
Similarly, we also use Eq.\,(\ref{eq.zurec}) to get two formulas of the function $Z_U$.
We conclude these four formulas in the following propositions.

\begin{proposition}
For any nonnegative integers $m,r$,
an $r$-tuple $\bm\alpha=(\alpha_1,\alpha_2,\ldots,\alpha_r)$ of positive integers, and
an $m$-tupe $\bm\beta=(\beta_1,\beta_2,\ldots,\beta_m)$ of positive integers,
with $\alpha_r\geq 2$, $\beta_m\geq 2$, we have
\begin{align}
\ZL{\bm\alpha}{\bm\beta}   \label{eq.2.6}
&=\sum_{a+b=m}(-1)^a\zeta(\beta_a,\ldots,\beta_1,\alpha_1,\ldots,\alpha_r)
\zeta^\star(\beta_{a+1},\ldots,\beta_m) \\
&=\sum_{a+b=r-1}(-1)^a\zeta(\alpha_{a+2},\ldots,\alpha_r)
\zeta^\star(\alpha_{a+1},\ldots,\alpha_1,\beta_1,\ldots,\beta_m).  \label{eq.2.7}
\end{align}
And
\begin{align}
\ZU{\bm\alpha}{\bm\beta}   \label{eq.2.8}
&=\sum_{a+b=m}(-1)^a\zeta(\alpha_1,\ldots,\alpha_r,\beta_m,\ldots,\beta_{b+1})
\zeta^\star(\beta_1,\ldots,\beta_b) \\
&=\sum_{a+b=r-1}(-1)^a\zeta(\alpha_1,\ldots,\alpha_b)
\zeta^\star(\beta_1,\ldots,\beta_m,\alpha_r,\ldots,\alpha_{b+1}), \label{eq.2.9}
\end{align}
where $\beta_i\geq 2\ (1\leq i\leq m)$ for Eq.\,(\ref{eq.2.8}) 
and $\alpha_j\geq 2\ (1\leq j\leq r)$ for Eq.\,(\ref{eq.2.9}).
\end{proposition}

We use a formula among the function $Z_B$ and $Z_U$
\cite[Proposition 7.1]{ChenE2022}
\begin{equation}\label{eq.zbu}
(-1)^m\ZB{\alpha_{m+1},\ldots,\alpha_r}{\alpha_m,\ldots,\alpha_1}
=\zeta(\bm\alpha)+\sum^m_{k=1}(-1)^k
\ZU{\alpha_{k+1},\ldots,\alpha_r}{\alpha_k,\ldots,\alpha_1},
\end{equation}
where $\alpha_r\geq 2$ and $1\leq m\leq r$,
and Eq.\,(\ref{eq.2.8}) to get 
an explicit formula for the function $Z_B$:
\begin{equation}\label{eq.2.10}
\ZB{\alpha_{m+1},\ldots,\alpha_r}{\alpha_m,\ldots,\alpha_1}
=\sum_{0\leq j\leq d\leq m}(-1)^{m+j-d}
\zeta(\alpha_{d+1},\ldots,\alpha_r,\alpha_1,\ldots,\alpha_j)
\zeta^\star(\alpha_d,\ldots,\alpha_{j+1}).
\end{equation}

\section{Symmetric Formulas}\label{sec.4}
We simplify our notations with 
$\oright{\alpha}{i}{j}=(\alpha_{i+1},\alpha_2,\ldots,\alpha_j)$, 
$\oleft{\alpha}{i}{j}=(\alpha_j,\alpha_{j-1},\ldots,\alpha_{i+1})$,
for $i<j$. Note that if $i\geq j$, we set 
$\oright{\alpha}{i}{j}=\oleft{\alpha}{i}{j}=\emptyset$. Therefore, the vector $\bm\alpha$ 
is $\oright{\alpha}{0}{r}$.
\begin{proposition}
For any nonnegative integers $m,r, \theta$,
an $r$-tuple $\bm\alpha=(\alpha_1,\alpha_2,\ldots,\alpha_r)$ of positive integers, and
an $m$-tupe $\bm\beta=(\beta_1,\beta_2,\ldots,\beta_m)$ of positive integers, we have
\begin{align} \label{eq.new01}
\sum^m_{k=0}(-1)^k\zeta(\oright{\beta}{k}{m})
\ZU{{\bm\alpha}}{\oleft{\beta}{0}{k}}
&=\zeta({\bm\alpha,\bm\beta}), \quad\mbox{for }\alpha_r\geq 2,\\
\sum^r_{k=0}(-1)^{r-k}\zeta(\oright{\alpha}{0}{k})         \label{eq.new02}
\ZL{\bm\beta}{\oleft{\alpha}{k}{r}}
&=\zeta(\bm\alpha,\bm\beta), \quad\mbox{for }\alpha_i,\beta_m\geq 2, 1\leq i\leq r,\\
\sum_{k=0}^m(-1)^k\zeta^\star(\oright{\beta}{k}{m})    \label{eq.new03}
\ZU{\oleft{\beta}{0}{k},\theta}{\bm\alpha}
&=\zeta^\star({\bm\alpha},\theta,{\bm\beta}), \quad\mbox{for }\theta,\beta_m\geq 2,\\
\sum^r_{k=0}(-1)^{r-k}\zeta^\star(\oright{\alpha}{0}{k})        \label{eq.new04}
\ZL{\theta,\oleft{\alpha}{k}{r}}{{\bm\beta}}
&=\zeta^\star({\bm\alpha},\theta,{\bm\beta}), \quad\mbox{for }\alpha_i,\theta,\beta_m\geq 2,
1\leq i\leq r.
\end{align}
\end{proposition}
\begin{proof}
We present the process to get first equation, the other equations are similarly to get.
Since $\zeta(\bm\alpha,\bm\beta)$ can be represented as 
\begin{align*}
I\left(
\begin{xy}
{(0,-14) \ar @{{*}-o} (4,-10)}, 
{(4,-10) \ar @{.o} (10,-4)}, 
{(10,-4) \ar @[red] @{{.}-} (14,0)}, 
{(14,0) \ar @{{*}.o} (20,6)}, 
{(20,6) \ar @{.o} (24,10)}, 
{(0,-13) \ar @/^2mm/ @{-}^{\bm\alpha} (9,-4)}, 
{(14,1) \ar @/^2mm/ @{-}^{\bm\beta} (23,11)} 
\end{xy}
\right)
&=
I\left(
\begin{xy}
{(0,-3) \ar@{{*}.o} |-{\bm\alpha} (8,5)},
{(0,-2) \ar@/^2mm/ @{-} (7,5)},
\end{xy}
\right)\cdot
I\left(
\begin{xy}
{(0,-3) \ar@{{*}.o} |-{\bm\beta} (8,5)},
{(0,-2) \ar@/^2mm/ @{-} (7,5)},
\end{xy}
\right)-I\left(
\begin{xy}
{(0,-10) \ar @{{*}-o} (4,-6)}, 
{(4,-6) \ar @{.o} (10,0)}, 
{(0,-9) \ar @/^2mm/ @{-}^{\bm\alpha} (9,0)}, 
{(10,0) \ar @[red] @{{.}-} (14,-4)}, 
{(14,-4) \ar @{{*}.o} (20,2)}, 
{(20,2) \ar @{.o} (24,6)}, 
{(14,-3) \ar @/^2mm/ @{-}^{\bm\beta} (23,7)} 
\end{xy}
\right)\\
&=\zeta(\bm\alpha)\cdot\zeta(\bm\beta)-
I\left(
\begin{xy}
{(0,-10) \ar @{{*}-o} (4,-6)}, 
{(4,-6) \ar @{.o} (10,0)}, 
{(10,0) \ar @[red] @{{.}-} (14,-4)}, 
{(14,-4) \ar @{{*}.o} (20,2)}, 
{(20,2) \ar @{.o} (24,6)}, 
{(0,-9) \ar @/^2mm/ @{-}^{\bm\alpha} (9,0)}, 
{(14,-3) \ar @/^2mm/ @{-}^{\bm\beta} (23,7)} 
\end{xy}
\right).
\end{align*}
Since for an integer $i$ with $1\leq i\leq m$, we have 
\begin{align*}
&I\left(
\begin{xy}
{(0,-12) \ar @{{*}-o} (4,-8)}, 
{(4,-8) \ar @{.o} (10,-2)}, 
{(0,-11) \ar @/^2mm/ @{-}^{\bm\alpha} (9,-2)}, 
{(10,-2) \ar @[red] @{{.}-} (12,-5)},
{(12,-5) \ar@{-} (18,1)},
{(12,-5) \ar@{{*}o} |-{\oleft{\bm\beta}{0}{i}} (30,1)},
{(12,-5) \ar@{.} (24,-5)},
{(18,1) \ar@{.} (30,1)},
{(24,-5) \ar@{-} (30,1)}, 
{(30,1) \ar @[red] @{{.}-} (33,4)},
{(33,4) \ar @{{*}-o} (37,8)}, 
{(37,8) \ar @{.o} (43,14)}, 
{(33,5) \ar @/^2mm/ @{-}^{\oright{\bm\beta}{i}{m}} (43,14)}
\end{xy}\right)\\
&=I\left(
\begin{xy}
{(0,-6) \ar @{{*}-o} (4,-2)}, 
{(4,-2) \ar @{.o} (10,4)}, 
{(0,-5) \ar @/^2mm/ @{-}^{\bm\alpha} (9,4)}, 
{(10,4) \ar @[red] @{{.}-} (12,1)},
{(12,1) \ar@{-} (18,7)},
{(12,1) \ar@{{*}o} |-{\oleft{\beta}{0}{i}} (30,7)},
{(12,1) \ar@{.} (24,1)},
{(18,7) \ar@{.} (30,7)},
{(24,1) \ar@{-} (30,7)}
\end{xy}\,\right)\cdot
I\left(
\begin{xy}
{(0,-5) \ar @{{*}-o} (4,-1)}, 
{(4,-1) \ar @{.o} (10,5)}, 
{(0,-4) \ar @/^2mm/ @{-}^{\oright{\bm\beta}{i}{m}} (10,5)}
\end{xy}\right)
-I\left(
\begin{xy}
{(0,-12) \ar @{{*}-o} (4,-8)}, 
{(4,-8) \ar @{.o} (10,-2)}, 
{(0,-11) \ar @/^2mm/ @{-}^{\bm\alpha} (9,-2)}, 
{(10,-2) \ar @[red] @{{.}-} (12,-5)},
{(12,-5) \ar@{-} (18,1)},
{(12,-5) \ar@{{*}o} |-{\oleft{\beta}{0}{i+1}} (34,1)},
{(12,-5) \ar@{.} (28,-5)},
{(18,1) \ar@{.} (34,1)},
{(28,-5) \ar@{-} (34,1)}, 
{(34,1) \ar @[red] @{{.}-} (37,4)},
{(37,4) \ar @{{*}-o} (41,8)}, 
{(41,8) \ar @{.o} (47,14)}, 
{(37,5) \ar @/^2mm/ @{-}^{\oright{\bm\beta}{i+1}{m}} (47,14)}
\end{xy}\right)\\
&=\ZU{\bm\alpha}{\oleft{\beta}{0}{i}}
\cdot \zeta(\oright{\beta}{i}{m})
-I\left(
\begin{xy}
{(0,-12) \ar @{{*}-o} (4,-8)}, 
{(4,-8) \ar @{.o} (10,-2)}, 
{(0,-11) \ar @/^2mm/ @{-}^{\bm\alpha} (9,-2)}, 
{(10,-2) \ar @[red] @{{.}-} (12,-5)},
{(12,-5) \ar@{-} (18,1)},
{(12,-5) \ar@{{*}o} |-{\oleft{\beta}{0}{i+1}} (34,1)},
{(12,-5) \ar@{.} (28,-5)},
{(18,1) \ar@{.} (34,1)},
{(28,-5) \ar@{-} (34,1)}, 
{(34,1) \ar @[red] @{{.}-} (37,4)},
{(37,4) \ar @{{*}-o} (41,8)}, 
{(41,8) \ar @{.o} (47,14)}, 
{(37,5) \ar @/^2mm/ @{-}^{\oright{\bm\beta}{i+1}{m}} (47,14)}
\end{xy}\right).
\end{align*}

Continue this process  we have
\begin{align*}
I\left(
\begin{xy}
{(0,-10) \ar @{{*}-o} (4,-6)}, 
{(4,-6) \ar @{.o} (10,0)}, 
{(10,0) \ar @[red] @{{.}-} (14,-4)}, 
{(14,-4) \ar @{{*}.o} (20,2)}, 
{(20,2) \ar @{.o} (24,6)}, 
{(0,-9) \ar @/^2mm/ @{-}^{\bm\alpha} (9,0)}, 
{(14,-3) \ar @/^2mm/ @{-}^{\bm\beta} (23,7)} 
\end{xy}\right)=\sum^m_{k=1}(-1)^{k-1}\ZU{\bm\alpha}{\oleft{\beta}{0}{k}}
\zeta(\oright{\beta}{k}{m}).
\end{align*}
Substituing into the original equation, we get the conclusion.
\end{proof}

\begin{theorem}\label{thm.41}
Let $r,m,\theta$ be nonnegative integers with $\theta\geq 2$, 
and the vectors $\bm\alpha=(\alpha_1,\alpha_2,\ldots,\alpha_r)$,
$\bm\beta=(\beta_1,\beta_2,\ldots,\beta_m)$ with $\alpha_i\geq 2$, $\beta_m\geq 2$,
for $1\leq i\leq r$. Then
\begin{align}    \label{eq.4.1}
\sum_{a+b=r\atop c+d=m}
(-1)^{a+c}\zeta^\star(\oright{\alpha}{0}{b})
\zeta(\oleft{\beta}{0}{c},\theta,\oleft{\alpha}{b}{r})
\zeta^\star(\oright{\beta}{c}{m})
&=\zeta^\star(\bm\alpha,\theta,\bm\beta),\\
\sum_{a+b=r\atop c+d=m}    \label{eq.4.2}
(-1)^{a+c}\zeta(\oright{\alpha}{0}{b})
\zeta^\star(\oleft{\beta}{0}{c},\theta,\oleft{\alpha}{b}{r})
\zeta(\oright{\beta}{c}{m})
&=\zeta(\bm\alpha,\theta,\bm\beta).
\end{align}
\end{theorem}
\begin{proof}
First we substitute the representation Eq.\,(\ref{eq.2.8}) of $Z_U$:
$$
\ZU{\oleft{\beta}{0}{k},\theta}{\bm\alpha}
=\sum_{a+b=r}(-1)^a\zeta(\oleft{\beta}{0}{k},\theta,\oleft{\alpha}{b}{r})
\zeta^\star(\oright{\alpha}{0}{b})
$$
into Eq.\,(\ref{eq.new03}), we have Eq.\,(\ref{eq.4.1}).
The other formula is obtained from Eq.\,(\ref{eq.new01}):
$$
\sum_{c+d=m}(-1)^c\zeta(\oright{\beta}{c}{m})
\ZU{\bm\alpha,\theta}{\oleft{\beta}{0}{c}}
=\zeta(\bm\alpha,\theta,\bm\beta),
$$
using the representation Eq.\,(\ref{eq.2.9}) of $Z_U$:
$$
\sum_{a+b=r}(-1)^a\zeta(\oright{\alpha}{0}{b})
\zeta^\star(\oleft{\beta}{0}{c},\theta,\oleft{\alpha}{b}{r}).
$$
\end{proof}

Another beautiful formulas derived from Theorem \ref{thm.41} 
are stated in the following results.

\begin{corollary}
For any nonnegative integers $p,q,s,t,\alpha$ with $s\geq 2,t\geq 2,\alpha\geq 2$, we have
\begin{align}
\sum_{a+b=q\atop c+d=p}
(-1)^{a+c}\zeta^\star(\{s\}^b)\zeta(\{s\}^a,\alpha,\{t\}^c)
\zeta^\star(\{t\}^d)&=\zeta^\star(\{t\}^p,\alpha,\{s\}^q),\\
\sum_{a+b=q\atop c+d=p}
(-1)^{a+c}\zeta(\{s\}^b)\zeta^\star(\{s\}^a,\alpha,\{t\}^c)
\zeta(\{t\}^d)&=\zeta(\{t\}^p,\alpha,\{s\}^q).
\end{align}
\end{corollary}
We note that a version of finite sums of the above two identities can be found in \cite{TY2013, Xu2019}.

\section{Duality Theorems}\label{sec.3}
\begin{theorem}
Let $\zeta(\bm\beta)$ be the dual of $\zeta(\bm\alpha)$ and $n$ be 
a nonnegative integer. Then 
\begin{equation}  \label{eq.3.2}
\ZU{\bm\alpha}{\{2\}^n}=\ZL{\bm\beta}{\{2\}^n}.
\end{equation}
\end{theorem}

In this section we will prove the duality theorems using three different methods.
\begin{description}
\item[Method $1$ --- The generating functions]
We let $G_{\bm\alpha}(x)$, $g_{\bm\alpha}(x)$ be the generating function of 
$\ZU{\bm\alpha}{\{2\}^n}$, $\ZL{\bm\alpha}{\{2\}^n}$, respectively, that is,
$$
G_{\bm\alpha}(x) = \sum^\infty_{n=0}\ZU{\bm\alpha}{\{2\}^n} x^{2n},
\quad\mbox{and}\quad
g_{\bm\alpha}(x) = \sum^\infty_{n=0}\ZL{\bm\alpha}{\{2\}^n} x^{2n}.
$$
We begin from the generating function $G_{\bm\alpha}(x)$.
\begin{align*}
G_{\bm\alpha}(x) 
&= \sum_{1\leq k_1<k_2<\cdots<k_r}
\frac{\prod_{1\leq j\leq k_r}\left(1-\frac{x^2}{j^2}\right)^{-1}}
{k_1^{\alpha_1}k_2^{\alpha_2}\cdots k_r^{\alpha_r}}\\
&=\frac{\pi x}{\sin(\pi x)}
\sum_{1\leq k_1<k_2<\cdots<k_r}
\frac{\prod_{j>k_r}\left(1-\frac{x^2}{j^2}\right)}
{k_1^{\alpha_1}k_2^{\alpha_2}\cdots k_r^{\alpha_r}}\\
&=\frac{\pi x}{\sin(\pi x)}\sum^\infty_{n=0}
(-1)^n\zeta(\alpha_1,\alpha_2,\ldots,\alpha_r,\{2\}^n) x^{2n}.
\end{align*}
Since the dual of $\zeta(\bm\alpha)$ is $\zeta(\bm\beta)$, we have
$$
\zeta(\bm\alpha,\{2\}^n)^{\dual}
=\zeta(\{2\}^n,\bm\beta).
$$
Therefore, the generating function $G_{\bm\alpha}(x)$ can be rewritten as
\begin{align*}
&\frac{\pi x}{\sin(\pi x)}
\sum_{1\leq \ell_1<\ell_2<\cdots<\ell_m}
\frac{\prod_{1\leq j<\ell_1}\left(1-\frac{x^2}{j^2}\right)}
{\ell_1^{\beta_1}\ell_2^{\beta_2}\cdots \ell_m^{\beta_m}}\\
&=\sum_{1\leq \ell_1<\ell_2<\cdots<\ell_m}
\frac{\prod_{\ell_1\leq j}\left(1-\frac{x^2}{j^2}\right)^{-1}}
{\ell_1^{\beta_1}\ell_2^{\beta_2}\cdots \ell_m^{\beta_m}}\ =\ 
g_{\bm\beta}(x).
\end{align*}
The coefficients of $x^{2n}$ of both generating functions give our identity.

\item[Method $2$ --- Yamamoto's integral]

Since the duality of the Euler sums is represented by $u_i=1-t_i$ 
in its Drinfel'd iterated integral, the corresponding $2$-poset Hasse diagram
appears as a vertical reflection, with $\circ$ and $\bullet$ interchanged.

\begin{align*}
\ZU{\bm\alpha}{\{2\}^m}=
I\left(\begin{xy}
{(-1,-3) \ar@{{*}.o} |-{\bm\alpha} (7,5)},
{(-1,-2) \ar@/^2mm/ @{-} (6,5)},
{(7,5) \ar@{-}@[red] (9,2)},
{(9,2) \ar@{-} (13,6)},
{(9,2) \ar@{{*}o} |-{\{2\}^m} (23,6)},
{(9,2) \ar@{.} (19,2)},
{(13,6) \ar@{.} (23,6)},
{(19,2) \ar@{-} (23,6)},
\end{xy}\ \right)
&\overset{\mbox{\tiny dual}}{=}
I\left(\,\begin{xy}
{(-1,5) \ar@{o.{*}} |-{\bm\beta} (7,-3)},
{(-1,4) \ar@/_2mm/ @{-} (6,-3)},
{(7,-2) \ar@{-}@[red] (8,0)},
{(8,0) \ar@{o-} (12,-4)},
{(12,-4) \ar@{} |-{\{2\}^m} (18,0)},
{(12,-4) \ar@{.} (22,-4)},
{(18,0) \ar@{-{*}} (22,-4)},
{(8,0) \ar@{.} (18,0)},
\end{xy}\ \right)\\
&=
I\left(\,\begin{xy}
{(-1,-4) \ar@{-} (3,0)},
{(-1,-4) \ar@{{*}o} |-{\{2\}^m} (13,0)},
{(-1,-4) \ar@{.} (9,-4)},
{(3,0) \ar@{.} (13,0)},
{(9,-4) \ar@{-} (13,0)},
{(13,0) \ar@{-}@[red] (15,-3)},
{(15,-3) \ar@{{*}.o} |-{\bm\beta} (22,4)},
{(15,-3) \ar@/^2mm/ @{-} (22,4)},
\end{xy}\ \right)
=\ZL{\bm\beta}{\{2\}^m}.
\end{align*}

We explain it by an example with
$\ZU{3,3}{2,2}^\#=\ZL{1,2,1,2}{2,2}$.
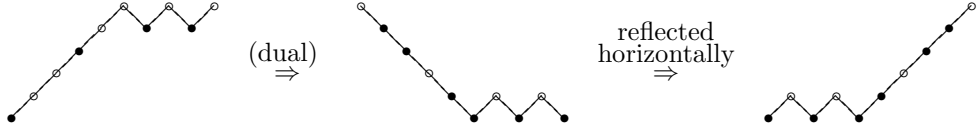
\[
\begin{xy}
{(0,-6) \ar@{{*}-} (3,-3)},
{(3,-3) \ar@{{o}-} (6,0)},
{(6,0) \ar@{o-} (9,3)},
{(9,3) \ar@{{*}-} (12,6)},
{(12,6) \ar@{o-} (15,9)},
{(15,9) \ar@{o-} (18,6)},
{(18,6) \ar@{{*}-} (21,9)},
{(21,9) \ar@{o-} (24,6)},
{(24,6) \ar@{{*}-o} (27,9)},
\end{xy}
\quad\overset{\mbox{(dual) }}{\Rightarrow}\quad
\begin{xy}
{(0,9) \ar@{o-} (3,6)},
{(3,6) \ar@{{*}-} (6,3)},
{(6,3) \ar@{{*}-} (9,0)},
{(9,0) \ar@{o-} (12,-3)},
{(12,-3) \ar@{{*}-} (15,-6)},
{(15,-6) \ar@{{*}-} (18,-3)},
{(18,-3) \ar@{o-} (21,-6)},
{(21,-6) \ar@{{*}-} (24,-3)},
{(24,-3) \ar@{o-{*}} (27,-6)},
\end{xy}
\quad\overset{\overset{\mbox{reflected}}{\mbox{horizontally}}}{\Rightarrow}\quad
\begin{xy}
{(0,-6) \ar@{{*}-} (3,-3)},
{(3,-3) \ar@{o-} (6,-6)},
{(6,-6) \ar@{{*}-} (9,-3)},
{(9,-3) \ar@{o-} (12,-6)},
{(12,-6) \ar@{{*}-} (15,-3)},
{(15,-3) \ar@{{*}-} (18,0)},
{(18,0) \ar@{o-} (21,3)},
{(21,3) \ar@{{*}-} (24,6)},
{(24,6) \ar@{{*}-o} (27,9)},
\end{xy}
\]

\item[Method $3$ --- Expressions by MZVs and MZSVs]

We use Eq.\,(\ref{eq.2.8}) and Eq.\,(\ref{eq.2.6}), we know that 
\begin{align*}
\ZU{\bm\alpha}{\{2\}^n}
&=\sum_{a+b=n}(-1)^a\zeta(\bm\alpha,\{2\}^a)\zeta^\star(\{2\}^b),\\
\ZL{\bm\alpha}{\{2\}^n}
&=\sum_{a+b=n}(-1)^a\zeta(\{2\}^a,\bm\beta)\zeta^\star(\{2\}^b).
\end{align*}
Since $\zeta(\bm\alpha)^\dual=\zeta(\bm\beta)$, this implies that 
$\zeta(\bm\alpha,\{2\}^a)^\dual=\zeta(\{2\}^a,\bm\beta)$. Thus, we have
\begin{align*}
\ZU{\bm\alpha}{\{2\}^n}
&=\sum_{a+b}\zeta(\bm\alpha,\{2\}^a)\zeta^\star(\{2\}^b) \\
&=\sum_{a+b}\zeta(\{2\}^a,\bm\beta)\zeta^\star(\{2\}^b) \
=\ \ZL{\bm\beta}{\{2\}^n}.
\end{align*}
\end{description}
At last we give two easy applications on this duality theorem.
We use Eq.\,(\ref{eq.2.10}) we know that 
$$
\ZB{\bm\alpha}{\{2\}^m}
=\sum_{a+b+c=m}(-1)^{a+b}\zeta(\{2\}^a,\bm\alpha,\{2\}^b)
\zeta^\star(\{2\}^c).
$$
It is easy to see that $\zeta(\{2\}^a,\bm\alpha,\{2\}^b)^\dual
=\zeta(\{2\}^b,\bm\beta,\{2\}^a)$, we have
\begin{align*}
\ZB{\bm\alpha}{\{2\}^n}
&=\sum_{a+b+c=n}(-1)^{a+b}\zeta(\{2\}^a,\bm\alpha,\{2\}^b)\zeta^\star(\{2\}^c) \\
&=\sum_{a+b+c=n}(-1)^{a+b}\zeta(\{2\}^b,\bm\beta,\{2\}^a)\zeta^\star(\{2\}^c) \ \ 
=\ \  \ZB{\bm\beta}{\{2\}^n}.
\end{align*}

The function $Z_B$ is also satisfied the duality property.
Nakasuji, Ohno \cite[Theorem 4.4]{NO2021} use Schur
multiple zeta functions to give a more general duality theorem.
However, our particular duality forms are founded by the classical generating functions,
or the new Yamamoto's integral.

Since $\zeta(k)^\dual=\zeta(\{1\}^{k-2},2)$, for $k\geq 2$. 
Then for any nonnegative integer $n$, we have
\begin{align*}
\zeta(2n+k) &= \ZB{k}{\{2\}^n} \ = \ \ZB{\{1\}^{k-2},2}{\{2\}^n} \\
&=\sum_{a+b+c=n}(-1)^{a+b}
\zeta(\{2\}^a,\{1\}^{k-2},\{2\}^{b+1})\zeta^\star(\{2\}^c).
\end{align*}
Let $k=2$, we have the following weighted sum formula:
\begin{equation}
\zeta(2n+2)=\sum_{a+b=n}(-1)^b(b+1)\zeta(\{2\}^{b+1})\zeta^\star(\{2\}^a).
\end{equation}

\section{Applications: Some More Formulas}\label{sec.5}

\begin{theorem}
For any nonnegative integers $p$, $q$, and $r$ with $p>0$, $q>0$, we have
\begin{align}\nonumber
&\zeta^\star(\{2\}^p,\{1\}^r,\{2\}^q)+(-1)^{r+1}\zeta^\star(\{2\}^q,r+2,\{2\}^{p-1}) \\
&=\sum_{a+b=r-1}(-1)^a\zeta^\star(a+2,\{2\}^{p-1})
\zeta^\star(\{1\}^{b+1},\{2\}^q). \label{eq.51}
\end{align}
\end{theorem}
\begin{proof}
Let 
$$
G(x)=\sum^\infty_{p=0}\zeta^\star(\{2\}^p,\{1\}^r,\{2\}^q)x^{2p}
$$
be the generating function of $\zeta^\star(\{2\}^p,\{1\}^r,\{2\}^q)x^{2p}$. Thus,
\begin{align*}
G(x) &= \sum_{1\leq k_1\leq k_2\leq \cdots\leq k_{q+r}}
\frac{\prod_{1\leq j\leq k_1}\left(1-\frac{x^2}{j^2}\right)^{-1}}
{k_1k_2\cdots k_rk_{r+1}^2k_{r+2}^2\cdots k_{q+r}^2} \\
&=\frac{\pi x}{\sin(\pi x)}
\sum_{1\leq k_1\leq k_2\leq \cdots\leq k_{q+r}}
\frac{\prod_{j>k_1}\left(1-\frac{x^2}{j^2}\right)}
{k_1k_2\cdots k_rk_{r+1}^2k_{r+2}^2\cdots k_{q+r}^2}.
\end{align*}
We represent the above summation as
$$
\zeta^\star(\{1\}^r,\{2\}^q)
+\sum^\infty_{n=1}(-1)^n\ZL{1,\{2\}^n}{\{1\}^{r-1},\{2\}^q}x^{2n}.
$$
By convolution, we have
\begin{align*}
G(x) &= \sum^\infty_{n=0}\zeta^\star(\{2\}^n)\zeta^\star(\{1\}^r,\{2\}^q)x^{2n} \\
&\quad+\sum^\infty_{p=1}
\left(\sum_{m+n=p\atop m\geq 1}(-1)^m\zeta^\star(\{2\}^n)
\ZL{1,\{2\}^m}{\{1\}^{r-1},\{2\}^q}\right)x^{2p}.
\end{align*}
Comparing the coefficient of $x^{2p}$, for $p\geq 1$, we have
\begin{align}\label{eq.xx}
\zeta^\star(\{2\}^p,\{1\}^r,\{2\}^q) &=
\zeta^\star(\{2\}^p)\zeta^\star(\{1\}^r,\{2\}^q) \\
&\quad+\sum_{m+n=p-1}(-1)^{m+1}\zeta^\star(\{2\}^n)
\ZL{1,\{2\}^{m+1}}{\{1\}^{r-1},\{2\}^q}. \nonumber
\end{align}
We use Eq.\,(\ref{eq.zlrec}) to the function $Z_L$, this gives 
\begin{align*}
\ZL{1,\{2\}^{m+1}}{\{1\}^{r-1},\{2\}^q}
&=(-1)^{r-1}\ZL{\{1\}^r,\{2\}^{m+1}}{\{2\}^q}\\
&\quad+\sum_{a+b=r-2}(-1)^a\zeta(\{1\}^{a+1},\{2\}^{m+1})
\zeta^\star(\{1\}^{b+1},\{2\}^q).
\end{align*}
It is clear that $\zeta(\{1\}^{a+1},\{2\}^{m+1})^\dual 
= \zeta(\{2\}^m,a+3)$, and we apply the dual theorem to transform
$$
\ZL{\{1\}^r,\{2\}^{m+1}}{\{2\}^q} =
\ZU{\{2\}^m,r+2}{\{2\}^q}
$$
then we obtain
\begin{align*}
\ZL{1,\{2\}^{m+1}}{\{1\}^{r-1},\{2\}^q}
&=(-1)^{r-1}\ZU{\{2\}^m,r+2}{\{2\}^q}\\
&\quad+\sum_{a+b=r-2}(-1)^a\zeta(\{2\}^m,a+3)
\zeta^\star(\{1\}^{b+1},\{2\}^q).
\end{align*}
The summation in Eq.\,(\ref{eq.xx}) becomes
\begin{align*}
&\sum_{m+n=p-1}(-1)^{m+r}\zeta^\star(\{2\}^n)
\ZU{\{2\}^m,r+2}{\{2\}^q} \\
&\quad+\sum_{m+n=p-1}\sum_{a+b=r-2}
(-1)^{m+a+1}\zeta^\star(\{2\}^n)\zeta(\{2\}^m,a+3)\zeta^\star(\{1\}^{b+1},\{2\}^q).
\end{align*}
The second summation of the above formula can be simplified by using Eq.\,(\ref{eq.2.3})
$$
\sum_{m+n=p-1}(-1)^m\zeta^\star(\{2\}^n)\zeta(\{2\}^m,a+3)
=\zeta^\star(a+3,\{2\}^{p-1}).
$$
We will get 
$$
\sum_{a+b=r-2}(-1)^{a+1}\zeta^\star(a+3,\{2\}^{p-1})\zeta^\star(\{1\}^{b+1},\{2\}^q).
$$
We remain to treat the summation
$$
\sum_{m+n=p-1}(-1)^{m+r}\zeta^\star(\{2\}^n)
\ZU{\{2\}^m,r+2}{\{2\}^q}.
$$
We apply Eq.\,(\ref{eq.new03}) with $\bm\alpha=(\{2\}^q)$, $\theta=r+2$, and $\bm\beta=(\{2\}^{p-1})$,
then the above summation is equal to 
$$
(-1)^r\zeta^\star(\{2\}^q,r+2,\{2\}^{p-1}).
$$
Therefore Eq.\,(\ref{eq.xx}) becomes
\begin{align*}
\zeta^\star(\{2\}^p,\{1\}^r,\{2\}^q) &=
\zeta^\star(\{2\}^p)\zeta^\star(\{1\}^r,\{2\}^q)+(-1)^r\zeta(\{2\}^q,r+2,\{2\}^{p-1})\\
&\quad+\sum_{a+b=r-2}(-1)^{a+1}\zeta^\star(a+3,\{2\}^{p-1})\zeta^\star(\{1\}^{b+1},\{2\}^q) \\
&=(-1)^r\zeta^\star(\{2\}^q,r+2,\{2\}^{p-1})\\
&\quad+\sum_{a+b=r-1}(-1)^a\zeta^\star(a+2,\{2\}^{p-1})\zeta^\star(\{1\}^{b+1},\{2\}^q).
\end{align*}
This finishes our work.
\end{proof}

It is well-known that (ref. \cite{Zlobin2005}) $\zeta^\star(1,\{2\}^q)=2\zeta(2q+1)$, 
and we leverage Zagier's formula (ref. \cite{Zagier2012}) to compute $\zeta^\star(\{2\}^q,3,\{2\}^{p-1})$:
$$
\zeta^\star(\{2\}^q,3,\{2\}^{p-1})=-2\sum^{p+q}_{k=1}
\left[\binom{2k}{2q}-\delta_{k,q}-\left(1-\frac{1}{2^{2k}}\right)\binom{2k}{2p-1}\right]
\zeta^\star(\{2\}^{p+q-k})\zeta(2k+1).
$$
By substituting $r=1$ in Eq.\,(\ref{eq.51}), we obtain an evaluation of $\zeta^\star(\{2\}^p,1,\{2\}^q)$, 
for any positive integers $p$ and $q$:
$$
\zeta^\star(\{2\}^p,1,\{2\}^q) =
2\sum^{p+q}_{k=1}
\left[\binom{2k}{2q}-\left(1-\frac{1}{2^{2k}}\right)\binom{2k}{2p-1}\right]
\zeta^\star(\{2\}^{p+q-k})\zeta(2k+1).
$$

On the other hand, if we apply $p=q=1$ in Eq.\,(\ref{eq.51}), then 
$$
\sum_{a+b=r}(-1)^a(b+1)\zeta(a+2)\zeta(b+2)
=\zeta^\star(2,\{1\}^r,2)+(-1)^{r}\zeta(r+2,2),
$$
for any nonnegative integer $r$, by using the fact (ref. \cite{ChenE2022})
$\zeta^\star(\{1\}^{b+1},2)=(b+2)\zeta(b+3)$.

\titleformat{\section}
{\sffamily\color{sectitlecolor}\Large\bfseries\filcenter}{}{2em}{#1}%

\end{document}